\documentclass[a4paper,12pt]{amsart}

\usepackage[mathcal]{euscript}   
\usepackage{mathrsfs}
\usepackage[shortlabels]{enumitem}
\usepackage[matrix,arrow]{xy}
\usepackage{comment}
\usepackage{amsmath}
\usepackage{amssymb,enumitem,verbatim,stmaryrd,xcolor,microtype,graphicx,needspace}
\usepackage[T1]{fontenc}
\usepackage[utf8]{inputenc}
\usepackage[english]{babel} 
\usepackage[top=3.3cm,bottom=3.3cm,left=3.25cm,right=3.25cm]{geometry}
\usepackage[bookmarksdepth=2,linktoc=page,colorlinks,linkcolor={red!80!black},citecolor={red!80!black},urlcolor={blue!80!black}]{hyperref}

\usepackage{tikz}\usetikzlibrary{matrix,arrows,decorations.markings}
\usepackage{tikz-cd}
\usetikzlibrary{automata,positioning,decorations.pathreplacing}
\newcommand\cprime\textquotesingle 

\usepackage[figuresleft]{rotating}
\usepackage{changepage}

\usepackage{resizegather}
\usepackage{multicol}
\tikzset{->-/.style={decoration={markings,mark=at position #1 with {\color{black}\arrow{>}}},postaction={decorate,very thick}}}
\tikzstyle{vertex}=[circle, draw, inner sep=0pt, minimum size=6pt]

\usepackage[mathcal]{euscript}                               
\usepackage{mathptmx}                                        
\usepackage[colorinlistoftodos,bordercolor=orange,backgroundcolor=orange!20,linecolor=orange,textsize=footnotesize]{todonotes} \makeatletter \providecommand \@dotsep{5} \def\listtodoname{List of Todos} \def\listoftodos{\@starttoc{tdo}\listtodoname} \makeatother 

\allowdisplaybreaks 

\usepackage{etoolbox}
\makeatletter
\patchcmd{\@startsection}{\@afterindenttrue}{\@afterindentfalse}{}{}             
\patchcmd{\part}{\bfseries}{\bfseries\LARGE}{}{}
\patchcmd{\section}{\scshape}{\bfseries}{}{}\renewcommand{\@secnumfont}{\bfseries} 
\patchcmd{\@settitle}{\uppercasenonmath\@title}{\large}{}{}
\patchcmd{\@setauthors}{\MakeUppercase}{}{}{}
\addto{\captionsenglish}{} 
\addto{\captionsenglish}{} 
\addto{\captionsenglish}{} 
\makeatother

\usepackage{fancyhdr}

\theoremstyle{plain}
\newtheorem{thm}{Theorem}[section]
\newtheorem{cor}[thm]{Corollary}
\newtheorem{lemma}[thm]{Lemma}
\newtheorem{prop}[thm]{Proposition}

\newtheorem*{thm*}{Theorem}
\newtheorem*{lem*}{Lemma}

\theoremstyle{definition}
\newtheorem{df}[thm]{Definition}
\newtheorem{rem}[thm]{Remark}

\newtheorem{ex}[thm]{Example}
\newtheorem*{df*}{Definition}
\newtheorem*{ex*}{Example}
\newtheorem*{rem*}{Remark}

\DeclareRobustCommand{\gobblefour}[5]{}    

\DeclareFontFamily{OT1}{pzc}{}                                
\DeclareFontShape{OT1}{pzc}{m}{it}{<-> s * [1.10] pzcmi7t}{}
\DeclareMathAlphabet{\mathpzc}{OT1}{pzc}{m}{it}
\DeclareSymbolFont{sfoperators}{OT1}{bch}{m}{n} \DeclareSymbolFontAlphabet{\mathsf}{sfoperators} \makeatletter\def\operator@font{\mathgroup\symsfoperators}\makeatother 
\DeclareSymbolFont{cmletters}{OML}{cmm}{m}{it}
\DeclareSymbolFont{cmsymbols}{OMS}{cmsy}{m}{n}
\DeclareSymbolFont{cmlargesymbols}{OMX}{cmex}{m}{n}
\DeclareMathSymbol{\myjmath}{\mathord}{cmletters}{"7C}     \let\jmath\myjmath 
\DeclareMathSymbol{\myamalg}{\mathbin}{cmsymbols}{"71}     
\DeclareMathSymbol{\mycoprod}{\mathop}{cmlargesymbols}{"60}
\DeclareMathSymbol{\myalpha}{\mathord}{cmletters}{"0B}     \let\alpha\myalpha 
\DeclareMathSymbol{\mybeta}{\mathord}{cmletters}{"0C}      \let\beta\mybeta
\DeclareMathSymbol{\mygamma}{\mathord}{cmletters}{"0D}     \let\gamma\mygamma
\DeclareMathSymbol{\mydelta}{\mathord}{cmletters}{"0E}     \let\delta\mydelta
\DeclareMathSymbol{\myepsilon}{\mathord}{cmletters}{"0F}   \let\epsilon\myepsilon
\DeclareMathSymbol{\myzeta}{\mathord}{cmletters}{"10}      \let\zeta\myzeta
\DeclareMathSymbol{\myeta}{\mathord}{cmletters}{"11}       \let\eta\myeta
\DeclareMathSymbol{\mytheta}{\mathord}{cmletters}{"12}     \let\theta\mytheta
\DeclareMathSymbol{\myiota}{\mathord}{cmletters}{"13}      \let\iota\myiota
\DeclareMathSymbol{\mykappa}{\mathord}{cmletters}{"14}     \let\kappa\mykappa
\DeclareMathSymbol{\mylambda}{\mathord}{cmletters}{"15}    \let\lambda\mylambda
\DeclareMathSymbol{\mymu}{\mathord}{cmletters}{"16}        \let\mu\mymu
\DeclareMathSymbol{\mynu}{\mathord}{cmletters}{"17}        \let\nu\mynu
\DeclareMathSymbol{\myxi}{\mathord}{cmletters}{"18}        \let\xi\myxi
\DeclareMathSymbol{\mypi}{\mathord}{cmletters}{"19}        \let\pi\mypi
\DeclareMathSymbol{\myrho}{\mathord}{cmletters}{"1A}       \let\rho\myrho
\DeclareMathSymbol{\mysigma}{\mathord}{cmletters}{"1B}     \let\sigma\mysigma
\DeclareMathSymbol{\mytau}{\mathord}{cmletters}{"1C}       \let\tau\mytau
\DeclareMathSymbol{\myupsilon}{\mathord}{cmletters}{"1D}   \let\upsilon\myupsilon
\DeclareMathSymbol{\myphi}{\mathord}{cmletters}{"1E}       \let\phi\myphi
\DeclareMathSymbol{\mychi}{\mathord}{cmletters}{"1F}       \let\chi\mychi
\DeclareMathSymbol{\mypsi}{\mathord}{cmletters}{"20}       \let\psi\mypsi
\DeclareMathSymbol{\myomega}{\mathord}{cmletters}{"21}     \let\omega\myomega
\DeclareMathSymbol{\myvarepsilon}{\mathord}{cmletters}{"22}\let\varepsilon\myvarepsilon
\DeclareMathSymbol{\myvartheta}{\mathord}{cmletters}{"23}  \let\vartheta\myvartheta
\DeclareMathSymbol{\myvarpi}{\mathord}{cmletters}{"24}     \let\varpi\myvarpi
\DeclareMathSymbol{\myvarrho}{\mathord}{cmletters}{"25}    \let\varrho\myvarrho
\DeclareMathSymbol{\myvarsigma}{\mathord}{cmletters}{"26}  \let\varsigma\myvarsigma
\DeclareMathSymbol{\myvarphi}{\mathord}{cmletters}{"27}    \let\varphi\myvarphi

\DeclareMathOperator{\Hom}{Hom}

\DeclareMathOperator{\GL}{GL}

\DeclareMathOperator{\rk}{rk}

\DeclareMathOperator{\Coh}{{Coh}}
\DeclareMathOperator{\Bun}{{Bun}}
\DeclareMathOperator{\PBun}{{\mathbb{P}Bun}}
\DeclareMathOperator{\Pic}{{Pic}}
\DeclareMathOperator{\Ext}{{Ext}}

\newcommand\C{{\mathbb C}}

\newcommand\FF{{\mathbb F}}
\newcommand\F{{\mathcal F}}
\newcommand\G{{\mathcal G}}

\newcommand\N{{\mathbb N}}

\renewcommand\P{{\mathbb P}}
\newcommand\Q{{\mathbb Q}}

\newcommand\Z{{\mathbb Z}}

\newcommand\cA{{\mathcal A}}

\newcommand\cG{{\mathcal G}}

\newcommand\cH{{\mathcal H}}

\newcommand\cK{{\mathcal K}}

\newcommand\cO{{\mathcal O}}

\newcommand{\E}{\mathcal E}
\newcommand{\Line}{\mathcal L}
\newcommand{\Fq}{\mathbb{F}_q}

\makeatletter
\DeclareRobustCommand\bigop[1]{%
  \mathop{\vphantom{\sum}\mathpalette\bigop@{#1}}\slimits@
}
\newcommand{\bigop@}[2]{%
  \vcenter{%
    \sbox\z@{$#1\sum$}%
    \hbox{\resizebox{\ifx#1\displaystyle.9\fi\dimexpr\ht\z@+\dp\z@}{!}{$\m@th#2$}}%
  }%
}
\makeatother

\renewcommand\geq{\geqslant}
\renewcommand\leq{\leqslant}

\renewcommand\emptyset\varnothing

\title{ Hecke modifications of vector bundles}

\author{Roberto Alvarenga}
\address{\rm Roberto Alvarenga, São Paulo State University (UNESP), São José do Rio Preto, Brazil.}
\email{roberto.alvarenga@unesp.br}

\author{Inder Kaur}
\address{\rm Inder Kaur, School of Mathematics and Statistics, University of Glasgow, G12 8QQ , UK}
\email{inder.kaur@glasgow.ac.uk}

\author{Leonardo Moço}
\address{\rm Leonardo Moço, Instituto de Ci\^encias Matem\'aticas e de Computa\c{c}\~ao - USP, S\~ao Carlos, Brazil.}
\email{leonardo.moco@icmc.usp.br}

\allowdisplaybreaks
\begin{document}

\begin{abstract} Hecke modifications of vector bundles have played a significant role in several areas of mathematics. They appear in subjects ranging from number theory to complex geometry. This article intends to be a friendly introduction to the subject. We give an overview of how Hecke modifications appear in the literature, explain their origin and their importance in number theory and classical algebraic geometry. Moreover, we report the progress made in describing Hecke modifications explicitly and why these explicit descriptions are important. We describe all the Hecke modifications of the trivial rank $2$ vector bundle over a closed point of degree $5$ in the projective line, as well as all the vector bundles over a certain elliptic curve, which admit a rank $2$ and degree $0$ trace bundle as a Hecke modification. This result is not present in existing literature.  
\end{abstract}

\maketitle

\section{Introduction}

Let $X$ be a smooth projective curve defined over an arbitrary field $\mathbb{F}$. Let $D$ be an effective divisor on $X$. Given $\E, \E'$ two vector bundles (locally free sheaves) of the same rank over $X$, roughly speaking, we say that $\E'$ is a Hecke modification of $\E$  
at $D$ if $\E'$ is contained in $\E$ (as a locally free sheaf) with $\E/\E'$ isomorphic to the structural sheaf at $D$.  In this essay, we suppose that $D$ is supported in a single closed point of $X$. 

Hecke modifications of vector bundles have been studied, at least implicitly, since Weil \cite{weil-38} and as we shall see, appear naturally in different branches of pure mathematics. The name \textit{Hecke} comes from their connection with number theory, as it is related to the action of Hecke operators on the space of automorphic forms, see \cite{harder-67}. Hence, it plays an important role in the geometric Langlands correspondence, proved by Drinfeld and Lafforgue for $\GL_n$ over global function fields, cf.\  \cite{drinfeld} and \cite{lafforgue-02}.
In the remarkable work \cite{tyurin-69}, Hecke modifications appear in Tyurin's parametrization of the set of fixed rank and degree vector bundles over a Riemann surface.

In number theory, the explicit descriptions of Hecke modifications allow us to explicitly calculate the action of Hecke operators on automorphic forms. This has been used to investigate the space of (unramified) automorphic forms and its subspaces spanned by eigenforms, cusp forms and toroidal forms. Among other things, one can use Hecke modifications to describe the support of cusp forms, compute its dimension and obtain a new proof for the Riemann Hypothesis for curves defined over a finite field. For more content in this direction, see \cite{oliver-toroidal} and \cite{oliver-elliptic}. When $\mathbb{F}$  is a local field, an analogous theory has been developed in \cite{frenkel-20} and \cite{etingof-frenkel-kazhdan-23}. Recently, May 2024, a team of nine mathematicians led by Dennis Gaitsgory and Sam Raskin, claimed a proof of the geometric Langlands correspondence in the characteristic zero case.

In algebraic geometry, Hecke modifications have been used as a tool for decades beginning with \cite{narasimhan-ramanan-78}, where they are used to study the moduli space of rank $2$ vector bundles with fixed determinant and in \cite{mehta-seshadri-80}, where a connection between Hecke modifications and the parabolic structure of a given vector bundle is established. This connection continues to be explored, as is  evident from recent works such as \cite{Araujo2021automorphisms}, \cite{alfaya-gomez-21} and \cite{he-walpuski-19}.

\subsection*{Intention and scope of this article} 
The purpose of this survey is to provide a friendly introduction to Hecke modifications of vector bundles and how they appear in various contexts in the literature. We consider vector bundles defined over a smooth projective curve defined over a field $\mathbb{F}$, which is most of the time, either the complex field or a finite field.   
The article is inspired by the (in progress) joint works between the authors.

In the second section,  we introduce the main subject of this work and describe its first properties. The third section is devoted to explaining the origin of the name \textit{Hecke modifications} and its connection with number theory. In the fourth section, we exhibit the importance of the calculation of explicit Hecke modifications in both cases where either $\mathbb{F}$ is a finite field or the complex field. When $\mathbb{F}$ is a finite field, we introduce the theory of Hall algebras and show how to use it to calculate the Hecke modifications of a given vector bundle. The last section is devoted to another interpretation of Hecke modifications of rank $2$ vector bundles over the complex numbers: the elementary transformations. 
Throughout the article, we focus our attention on $\GL_n$ bundles over a smooth projective curve defined over the field $\mathbb{F}$. However, most of the definitions and results also work if we replace $\GL_n$  by $G$ a (semisimple) reductive connected algebraic group. 

\section{The name of the game} 

Let $X$ be a smooth projective curve defined over an arbitrary field $\mathbb{F}$. By $|X|$ we
denote the set of closed points of $X$. Given $x \in |X|$ and $ r \in \N_{\geq 1}$, let $\cK_{x}^{\oplus r}$ be the skyscraper sheaf supported at $x$ with stalk $\FF(x)^{\oplus r}$, where $\FF(x)$ is the residue field of $x$. We denote by $\mathrm{Coh} X$ the category of coherent sheaves on $X$ and by $\mathrm{Bun}_n X$ the set of isomorphism classes of rank $n$ vector bundles (of any degree) on $X$.  

We shall consider vector bundles on $X$ as locally free sheaves on $X$. Hence, we consider $\mathrm{Bun}_n X$ to be embedded in the category $\Coh (X)$ of coherent sheaves over $X$, cf.\  \cite[Ex. II.5.18]{hartshorne}. We denote $\Bun_1 X$ by $\Pic X$, which is an abelian group with tensor product as the group operation.

For a fixed $\F \in \Coh (X)$,  two exact sequences of coherent sheaves on $X$
\[0 \longrightarrow \F_1 \longrightarrow \F \longrightarrow \F_{2} \longrightarrow 0 \hspace{0.2cm}\text{ and } \hspace{0.2cm} 0 \longrightarrow \F_{1}' \longrightarrow \F \longrightarrow \F_2' \longrightarrow 0\]  
are isomorphic, if there are isomorphisms $\F_1 \rightarrow \F_1'$ and $\F_2 \rightarrow \F_2'$ such that 
$$\xymatrix@R5pt@C7pt{ 0 \ar[rr] && \F_1 \ar[rr] \ar[dd]^{\cong} && \F \ar[rr] \ar@{=}[dd] && \F_2 \ar[rr] \ar[dd]^{\cong} && 0 \\
&& && && && \\
0 \ar[rr] && \F_1' \ar[rr] && \F \ar[rr] && \F_2' \ar[rr] && 0 } $$
commutes. 

\begin{df} If $\E,\E' \in \mathrm{Bun}_n X$, $x \in |X|$ and $r \in \N_{\geq 1}$,
we denote by $[\E' \xrightarrow[r]{x} \E]$ the isomorphism class of a short exact sequence
\[0 \longrightarrow \E' \longrightarrow \E \longrightarrow \cK_{x}^{\oplus r} \longrightarrow 0. \]
 Moreover, we say that \emph{$\E'$ is a Hecke modification of $\E$} (or \emph{$\E$ is Hecke modified in $\E'$}) at $x$ with weight $r$. When $r=1$, we denote $[\E' \xrightarrow[r]{x} \E]$ simply by $[\E' \xrightarrow[]{x} \E]$. 
\end{df}

\begin{rem}
     Let $\E,\E' \in \mathrm{Bun}_n X$, $x \in |X|$ and $r \in \N_{\geq 1}$. The requirement that $\E'$ is a Hecke modification of $\E$ at $x$ with weight $r$, is equivalent to having a morphism $\xi: \E' \rightarrow \E$ which satisfies the following conditions:
     \begin{enumerate}
         \item the induced maps on stalk level $\xi_y: \E_{y}' \rightarrow \E_{y}$ are isomorphisms for every $y \in X $, $y \neq x$; and
         \item close to $x$, choosing bases $\E_x \cong \cO_{X,x}^{\oplus n}$ and $\E_x' \cong \cO_{X,x}^{\oplus n}$, we must have the following commutative diagram 
         \[\xymatrix@R5pt@C7pt{  \E_x' \ar[rr]^{\xi_x} \ar[dd]^{\cong} && \E_x \ar[dd]^{\cong} \\
&& \\
\cO_{X,x}^{\oplus n} \ar[rr]^{\overline{\xi}_x} && \cO_{X,x}^{\oplus n}  } \]
where 
\[ \overline{\xi}_x = 
\begin{pmatrix}
I_{n-r} & \\
 & \pi_x I_r 
\end{pmatrix}\]
$I_k$ is the $k \times k$ identity matrix (the empty entries in a matrix means zero entries) and $\pi_x \in \cO_{X,x}$ is an uniformizer at $x$.
\end{enumerate}
Observe that $(1)$ and $(2)$ ensure that $\xi$ is injective, since it is injective at level of stalks, cf.\ \cite[Prop. 1.1]{hartshorne}. 
\end{rem}

\begin{df} Let $\mathrm{H}_{x}^{r}(X)$ be the set of isomorphism classes of Hecke modifications in $X$ at $x$ of weight $r$, i.e., 
\[ \mathrm{H}_{x}^{r}(X) := \big\{ [\E' \xrightarrow[r]{x} \E] \;\big|\; \E,\E' \in \Bun_n X\big\}.\]
\end{df}

\begin{rem} \label{rem-heckestack}
    The above set may be seen as the set of $\mathbb{F}$-rational points of the \textit{Hecke stack}. We do not  give here the definition of Hecke stack, since we would like to keep this survey as accessible as possible. The interested reader may consult \cite[Sec. 3.2 of Chap. 12]{dennis-03}.  
\end{rem}

    Note that there are two natural maps from $\mathrm{H}_{x}^{r}(X)$ to $\Bun_n X$:  
 \begin{eqnarray*}
     h^{\leftarrow}: \quad    & \mathrm{H}_{x}^{r}(X)    
 & \longrightarrow \quad \Bun_n X \\
        & [\E' \xrightarrow[r]{x} \E]    & \longmapsto  \quad \E'        \end{eqnarray*}

and
 \begin{eqnarray*}
    h^{\rightarrow}: \quad &  \mathrm{H}_{x}^{r}(X) &  \longrightarrow \quad  \Bun_n X \\
      & [\E' \xrightarrow[r]{x} \E]  & \longmapsto \quad \E.
 \end{eqnarray*}
Let $[\E' \xrightarrow[r]{x} \E] \in \mathrm{H}_{x}^{r}(X)$, passing to the stalk at $x$ we obtain 
\[ 0 \rightarrow \E_{x}' \rightarrow \E_{x} \rightarrow \mathbb{F}(x)^{\oplus r} \rightarrow 0.\]
The above short exact sequence is no longer injective on the fiber at $x$ (i.e., after tensoring by $\mathbb{F}(x)$). Thus, we may write the  restriction to the fiber at $x$ as
\[ 0 \rightarrow \ker(\E_{x}' \otimes \mathbb{F}(x) \rightarrow \E_{x} \otimes \mathbb{F}(x) ) \rightarrow \E_{x}' \otimes \mathbb{F}(x) \rightarrow \E_{x} \otimes \mathbb{F}(x) \rightarrow \mathbb{F}(x)^{\oplus r} \rightarrow 0.\]
The previous construction sends $[\E' \xrightarrow[r]{x} \E]$ to a subspace of dimension $r$ in $\E_{x}' \otimes \mathbb{F}(x) $ and to a subspace of dimension $n-r$ in 
$\E_{x} \otimes \mathbb{F}(x).$ 

Conversely, given a subspace $V \subset \E_{x}' \otimes \mathbb{F}(x)$ of dimension $r$, we define $\E$ to be the subsheaf of $\E'(x) := \E' \otimes \mathcal{O}_X(x)$ whose set of sections over an open set $U \subseteq X$ is given by
\[\E(U) := \big\{ s \in \E'(x)(U) \;|\; s= \pi_x^{-1}t, t \in \E'(U), \text{ and if } x \in U, \text{ then } t(x) \in V \big\}\]
where $\pi_x$ is an uniformizer on  $X$ at $x$. Thus, we obtain 
\[0 \longrightarrow \E' \longrightarrow \E \longrightarrow \cK_{x}^{\oplus r} \longrightarrow 0, \]
a short exact sequence of coherent sheaves.

On the other hand, given a subspace $W \subset \E_{x} \otimes \mathbb{F}(x)$ of dimension $n-r$, we define $\E'$ to be the subsheaf of $\E$ whose set of sections over an open set $U \subseteq X$ is given by
\[\E'(U) := \big\{ s \in \E(U) \;|\; \text{ if } x \in U, \text{ then } s(x) \in W \big\}.\]
Hence we also obtain 
\[0 \longrightarrow \E' \longrightarrow \E \longrightarrow \cK_{x}^{\oplus r} \longrightarrow 0 \]
a short exact sequence of coherent sheaves. The above discussion yields the following. 

\begin{thm} \label{thm-grass}
    Let $\E,\E' \in \Bun_n X$, $x \in |X|$ and $r \in \Z,\; 0 < r \leq n$. The fibers of $h^{\rightarrow}$ over $\E$ can be canonically identified with the set of dimension $n-r$ subspaces inside the $\mathbb{F}(x)$-vector space $\E_x \otimes \mathbb{F}(x)$. In other words, 
    $(h^{\rightarrow})^{-1}(\E)$ is the set of $\mathbb{F}(x)$-rational points of the Grassmannian $\mathrm{Gr}_{n-r}(\E_x).$ The same is true for $h^{\leftarrow}$, in which case we obtain that $(h^{\leftarrow})^{-1}(\E')$ is the space of $\mathbb{F}(x)$-rational points of the Grassmannian $\mathrm{Gr}_{r}(\E_x').$
\end{thm}   

Choosing bases for $\E_x \cong \mathcal{O}_{X,x}^{\oplus n},$ and 
$\E_x' \cong \mathcal{O}_{X,x}^{\oplus n}$ we obtain a noncanonical version of the previous theorem. 

\begin{thm} \label{thm-grass}
    Let $\E,\E' \in \Bun_n X$, $x \in |X|$ and $r \in \Z,\; 0 < r \leq n$. The fibers $(h^{\rightarrow})^{-1}(\E)$ can be non-canonically identified with the set of $\mathbb{F}(x)$-rational points of the Grassmannian $\mathrm{Gr}_{n-r}(\mathbb{F}(x)^{\oplus n}).$ Moreover, $(h^{\leftarrow})^{-1}(\E')$ can be non-canonically identified with the space of $\mathbb{F}(x)$-rational points of the Grassmannian $\mathrm{Gr}_{r}(\mathbb{F}(x)^{\oplus n}).$
\end{thm}   

\begin{df} \label{defHeckemodinfixedvb}
    For a fixed $\E \in \Bun_n X$, we define the space of Hecke modifications $\mathrm{H}_{x}^{r}(X, \E)$ in $\E$ at $x$ of weight $r$ to be 
    the set $(h^{\rightarrow})^{-1}(\E)$ i.e., 
    \[ \mathrm{H}_{x}^{r}(X, \E) := \big\{ [\E' \xrightarrow[r]{x} \E] \;\big|\; \E' \in \Bun_n X \big\}.\]
 \end{df}

\begin{rem}
    The above theorem tell us that $\mathrm{H}_{x}^{r}(X, \E)$  has the structure of an algebraic variety (a complex manifold in the case $\mathbb{F}=\C$) that is non-canonically isomorphic to the Grassmannian $\mathrm{Gr}_{n-r}(\mathbb{F}(x)^{\oplus n}).$
\end{rem}

\section{The origin} 

In this section we assume $\mathbb{F}= \Fq$ to be the finite field with $q$ elements, where $q$ is a prime power. As before, $X$ is a geometrically irreducible smooth projective curve defined over $\mathbb{F}= \Fq$.
Let $\E,\E' \in \Bun_n X$ and $\cK_{x}^{\oplus r}$ be the skyscraper sheaf supported at $x$. We now discuss why such equivalence classes  
\[ 0 \rightarrow \E' \rightarrow \E \rightarrow \cK_{x}^{\oplus r} \rightarrow 0\]
of short exact sequences are called Hecke modifications!

Let $F := \Fq(X)$ be the function field of $X$ and $g(X)$ denote the genus of $X$. As in the previous section, $|X|$ is the set of closed points of $X$ or, equivalently, the set of places in $F$.
For $x \in |X|$, we denote by $F_x$ the completion of $F$ at $x$, by $\mathcal{O}_x$ the ring of integers of $F_x$, by $\pi_x \in \mathcal{O}_x$ (we can assume $\pi_x \in F$) an uniformizer of $x$ and by $q_x$ the cardinal of the residue field $\mathbb{F}(x):=\mathcal{O}_x/(\pi_x) \cong \mathbb{F}_{q_x}.$ Moreover, we denote by $|x|$  the degree of $x$, which is defined by the extension field degree $[\mathbb{F}(x) : \Fq]$. In other words, $q_x = q^{|x|}$. 
Let $| \cdot |_x$ be the absolute value of $F_x$ (resp. $F$) such that $|\pi_x|_x = q_{x}^{-1},$ we call $| \cdot |_x$ the local norm for each $x \in |X|$.

Let $\mathbb{A}$ be the adele ring of $F$ and $\mathbb{A}^{\times} $ the idele group. We denote by $\mathcal{O}_{\mathbb{A}} := \prod \mathcal{O}_x $ where the product is
taken over all places $x$ of $F$. We may consider $F_x$ as embedded into the adele ring $\mathbb{A}$, by sending an element $a \in F_x$ to the adele $(a_y)_{y \in |X|}$ with $a_x = a$ and $a_y = 0$ for $y \neq x$. Let $G(\mathbb{A}):= \mathrm{GL}_n(\mathbb{A})$, $Z$ be the center of  $G(\mathbb{A})$, $G(F):= \mathrm{GL}_n(F)$,  and $K:= \mathrm{GL}_n(\mathcal{O}_{\mathbb{A}})$ the standard maximal compact open subgroup of $G(\mathbb{A})$. For $x \in |X|$, consider $G_x=\GL_n(F_x)$ and $Z_{x}$ the center of $G_x$. Note that $G(\mathbb{A})$ comes together with the adelic topology that makes $G(\mathbb{A})$ into a locally compact group. Hence, $G(\mathbb{A})$ is endowed with a Haar measure. We fix the Haar measure on $G(\mathbb{A})$ for which $\mathrm{vol}(K)=1.$ The topology of $G(\mathbb{A})$ has a neighborhood basis $\mathcal{V}$ of the identity matrix that is given by all subgroups
\[K' = \prod_{x \in |X|} K_{x}' < \prod_{x \in |X|}K_x = K\]
where $K_x := \mathrm{GL}_n(\mathcal{O}_x)$, such that for all $x \in |X|$ the subgroup $K_{x}'$ of $K_x$ is open and consequently of finite index and such that $K_{x}^{'}$ differs from $K_x$ only for a finite number of places.

Let $C^0(G(\mathbb{A}))$ be the space of continuous and $\C$-valuated functions on $G(\mathbb{A})$. 
Roughly speaking, an automorphic form is a function $f \in C^0(G(\mathbb{A}))$ that can be identified as a function on the double coset $G(F) \setminus G(\mathbb{A})/Z K'$ which is of \textit{moderate growth}, for some $K' \in \mathcal{V}$. We refer the reader to \cite{bump} or \cite[Chap. 1]{oliver-thesis} for an exhaustive definition of an automorphic form and its properties.

We shall focus our attention when $K' = K$ in the above, i.e., on the automorphic forms defined over  
$G(F) \setminus G(\mathbb{A})/Z K$. These are the so-called \emph{unramified automorphic forms}. We denote the whole space of unramified automorphic forms by $\cA^K$. 
The other main actor of our story is the Hecke operator, which we introduce in the following.  A function in $C^0(G(\mathbb{A}))$ is called \emph{smooth} if it is locally constant. 

\begin{df} \label{def-hecke} The complex vector space $\mathcal{H}$ of all smooth compactly supported functions $\Phi : G(\mathbb{A}) \rightarrow \C$ together with the convolution product
$$\Phi_1 \ast \Phi_2: g \longmapsto \int_{G(\mathbb{A})} \Phi_1(gh^{-1})\Phi_2(h)dh$$
for $\Phi_1, \Phi_2 \in \mathcal{H}$ is called the \emph{Hecke algebra} for $G(\mathbb{A})$. Its elements are called \emph{Hecke operators}.
\end{df}

The zero element of $\mathcal{H}$ is the zero function, but there is no multiplicative unit. 
For $K' \in \mathcal{V},$ we define $\mathcal{H}_{K'}$ to be the subalgebra of all (left and right) bi-$K'$-invariant elements. These subalgebras have multiplicative units. Namely, the normalized characteristic function $\epsilon_{K'} := (\mathrm{vol}K')^{-1} \mathrm{char}_{K'}$ acts as the identity on $\mathcal{H}_{K'}$ by convolution. 

 When $K'=K$ in the above, we call $\mathcal{H}_K$ the unramified (or spherical) part of $\mathcal{H}$, and its elements are called \emph{unramified (or spherical) Hecke operators}. For $K' \in \mathcal{V},\ K'\neq K$, $\mathcal{H}_{K'}$ is called the ramified part of $\mathcal{H}$, and its elements are called   ramified Hecke operators.
It is well known that every $\Phi \in \mathcal{H}$ is bi-$K'$-invariant for some $K' \in \mathcal{V}$, cf.\ \cite[Prop. 1.4.4]{oliver-thesis}. In particular, 
$\mathcal{H} = \bigcup_{K' \in \mathcal{V}} \mathcal{H}_{K'}.$

 The unramified (or spherical) Hecke algebra $\mathcal{H}_K$ has the following characterization.  For $x \in |X|$, let $\Phi_{x,r}$ be the characteristic function in
\[K\begin{pmatrix}
    \pi_x I_r &  \\ 
 & I_{n-r}
\end{pmatrix} K\] 
and $\Phi_{x,n}$ be the characteristic function in $K \pi_x I_n K = \pi_x I_n K$,
where $I_k$ denotes the $k \times k$ identity matrix. 
Both $\Phi_{x,r}$ and $\Phi_{x,0}$ are elements of $\cH_K$. A theorem of Satake, cf.\ \cite[Thm. 4.6.1 and Prop. 4.6.1]{bump} or \cite{satake63}, states that these Hecke operators generate the unramified Hecke algebra. 

\begin{thm} Identifying the characteristic function on $K$ with $1 \in \C$ yields 
\[ \cH_K \cong \C[\Phi_{x,1}, \ldots, \Phi_{x,n}, \Phi_{x,n}^{-1}]_{x \in |X|}.\] 
In particular, $\cH_K$ is commutative. 
\end{thm}

The Hecke algebra $\mathcal{H}$ acts on $C^{0}(G(\mathbb{A}))$ as
\[\Phi(f): g \longmapsto \int_{G(\mathbb{A})} \Phi(h) f(gh) dh.\]
The action of the Hecke algebra on $C^{0}(G(\mathbb{A}))$ descends to an action on the space of automorphic forms. In particular, the unramified Hecke algebra $\mathcal{H}_K$ acts on the space of unramified automorphic forms $\cA^K$. By the Satake theorem, the last action is completely determined by $\Phi_{x,r}(f)$, where $f \in \cA^K$ and $r=1,\ldots,n.$ In what follows, we will see $\cA^K$ as the space of functions on $\Bun_n X$ and thus, the Hecke modifications at $x \in |X|$ with weight $r$, precisely describe the action of $\Phi_{x,r}(f)$.

\begin{rem}
    
One might also consider automorphic forms and Hecke operators in classical context and over any global field, cf. \cite[Sec. 3.1, 3.2]{bump}. 
The Hecke operators and their action on automorphic forms play a key role in the context of both classical and modern number theory. For instance, the Modularity Theorem  states that there is an automorphic form attached to each rational elliptic curve and moreover, that this automorphic form is actually a Hecke eigenform.  Also known as Taniyama–Shimura-Weil conjecture, the Modularity Theorem was proved by Wiles in \cite{wiles} (with a key step given by joint work with Taylor \cite{taylor-wiles-95}) for semi-stable elliptic curves, completing the proof of the Fermat Last Theorem after more than 350 years. The Modularity Theorem was proved completely by Breuil, Conrad, Diamond and Taylor in \cite{bcdt-01}, see also \cite{diamond-shurman-05}.  More generally, the Hecke operators and their action on the space of automorphic forms play a central role in the geometric Langlands correspondence. 
 \end{rem}

We now return to the geometry of $X$ and explain how the Hecke modifications are connected with the action of (unramified) Hecke operators over (unramified) automorphic forms.  The connection between the two worlds is given by the following theorem of Weil. 

The bijection
\[\begin{array}{ccccc}
F^{*}\setminus \mathbb{A}^{*} / \mathcal{O}_{\mathbb{A}}^{*} & = \mathrm{Cl}F & \stackrel{1:1}{\longleftrightarrow} & \mathrm{Pic}X = & \mathrm{Bun}_{1}X \\
{[a]} & & \longmapsto & & \Line_{a}  \end{array} \] 
where $\Line_{a} = \Line_{D}$ if $D$ is the divisor determined by $a$, generalizes to all vector bundles as follows.
A rank $n$ bundle $\E$ can be described by choosing bases for all stalks
\[\E_{\eta} \cong \mathcal{O}_{X, \eta}^{n} = F^{n}, \quad \quad \E_x \cong \mathcal{O}_{X,x}^{n} = (\mathcal{O}_x \cap F)^{n}\]
and inclusion maps $\E_x \hookrightarrow \E_{\eta}$,
where $\eta$ is the generic point of $X$ and $x$ runs over all closed point of $X$.  After tensoring with $F_x$, for each $x \in |X|$, we obtain
\[ F_{x}^{n} \cong \mathcal{O}_{x}^{n} \otimes_{\mathcal{O}_{x}} F_{x} \cong \E_{x} \otimes_{\mathcal{O}_{X,x}} \mathcal{O}_{x} \otimes_{\mathcal{O}_{x}} F_{x} \cong \E_{x} \otimes_{\mathcal{O}_{X,x}} F \otimes_{F} F_{x} \cong F^{n} \otimes_{F} F_{x} \cong F_{x}^{n} \]
which yields an element $g$ of $G(\mathbb{A}).$ 
A change of bases for $\E_{\eta}$ and $\E_x$ corresponds to multiplying $g$ by an element of $G(F)$ from the left and by an element of $K$ from the right, respectively.

Since the inclusion $F \subset F_x$ is dense for every place $x$, and $G(\mathcal{O}_{\mathbb{A}})$  is open in $G(\mathbb{A})$,
every class in $G(F)\setminus  G(\mathbb{A})/G(\mathcal{O}_{\mathbb{A}})$ is represented by a $g= (g_x) \in G(\mathbb{A})$ such that
$g_x \in G(F)$ for all $x \in |X|$. This means that the above construction can be reversed.  
We refer \cite[Parag. 5.1.5]{oliver-thesis} for the precise description of the above construction. Weil's theorem asserts the following.

\begin{thm}\cite[Lemma 3.1]{frenkel-04}   \label{weil-thm} For every $n \geq 1,$ the above construction yields a bijection
\[\begin{array}{ccc}
G(F) \setminus G(\mathbb{A}) / K & \stackrel{1:1}{\longleftrightarrow} & \mathrm{Bun}_n X \\ 
g & \longmapsto & \E_g.
\end{array}\]
Moreover, if $\PBun_n X$ is the set of isomorphism classes of $\P^{n-1}$-bundles over $X$ (cf. \cite[Ex. II.7.10]{hartshorne}),  then the above map induces
\[\begin{array}{ccc}
G(F) \setminus G(\mathbb{A}) / Z K & \stackrel{1:1}{\longleftrightarrow} & \P \mathrm{Bun}_n X \\ 
\overline{g} & \longmapsto & \overline{\E_g}.
\end{array}\]
a bijection as well, where  $\overline{g}$ (resp. $\overline{\E_g}$) stands for its class in $G(F) \setminus G(\mathbb{A}) / Z K$  (resp. $\P \mathrm{Bun}_n X$).
\end{thm}

By the above theorem,  we may regard the space $\cA^K$ of unramified automorphic forms as the space of complex valued functions on $\P\Bun_n X$. 
Hence, from now on we consider $f \in \cA^K$ as a function
\[ f : \P\Bun_n X \longrightarrow \C.\] 
As we shall see below, as $\E$ runs over $\Bun_n X$, its Hecke modifications (at $x \in |X|$, of weight $r$) describe the action of $\Phi_{x,r}$ on $\cA^K$. We begin by explaining that it is compatible with the domain of $f \in \cA^K$ being the set $\P\Bun_n X$. We might consider $\P\Bun_n X$ as the orbit set $\Bun_n X/\Pic X$ under the following action 
\[\begin{array}{ccc}
    \Bun_n X \times \Pic X & \longrightarrow & \Bun_n X. \\
     (\E, \Line) & \longmapsto & \E \otimes \Line  
\end{array} \]
For $\E \in \Bun_n X$, we denote by $\overline{\E}$ its class in $\PBun_n X$. 

Let $\Line \in \Pic X$, since the tensor by a line bundle $\Line \in \Pic X$ is an exact functor, we have the bijection
\[\begin{array}{ccc}
   \left\{ \substack{ \text{ isomorphism classes}  \\ 
   0 \longrightarrow \E' \longrightarrow \E \longrightarrow \cK_{x}^{\oplus r} \longrightarrow 0 \\
   \text{  with fixed } \E } \right\}
   
   & \longrightarrow &    
    \left\{ \substack{ \text{ isomorphism classes}  \\ 
   0 \longrightarrow \E'' \longrightarrow \E \otimes \Line \longrightarrow \cK_{x}^{\oplus r} \longrightarrow 0 \\
   \text{  with fixed } \E } \right\} \vspace*{.25cm} \\ 
   
     {\tiny (0 \rightarrow \E' \rightarrow \E \rightarrow \cK_{x}^{\oplus r} \rightarrow 0)} & \longmapsto &
     {\tiny (0 \rightarrow \E'\otimes \Line \rightarrow \E \otimes \Line  \rightarrow \cK_{x}^{\oplus r} \rightarrow 0)} 
\end{array} \]
With the notation of the Definition \ref{defHeckemodinfixedvb}, the above bijection translates to 
\[\begin{array}{ccc}
    \mathrm{H}_{x}^{r}(X, \E) & \longrightarrow & \mathrm{H}_{x}^{r}(X, \E \otimes \Line)  \vspace*{.15cm} \\
    { [\E' \xrightarrow[r]{x} \E]} & \longmapsto & 
    { [\E' \otimes \Line \xrightarrow[r]{x} \E \otimes \Line]. }
\end{array}\]
 Therefore, for the purposes of understanding the action $\Phi_{x,r}$ on $\cA^K$, describe the Hecke modifications of $\E \in \Bun_n X$ is compatible with considering $f \in \cA^K$ as a function on either $\Bun_n X$ or $\PBun_n X.$   

In this context, the action of $\cH_K$ over $\cA^K$ is given as follows. Recall from the previous section that we have the maps
\[ \Bun_n X \xleftarrow[]{h^{\leftarrow}} \mathrm{H}_{x}^{r}(X)  \xrightarrow[]{h^{\rightarrow}}  \Bun_n X\]
Thus consider the operator on  $\cA^K$ given by
\[ f \longmapsto (h^{\leftarrow})_{!}(h^{\rightarrow})^{*}(f)\]
where the superscript $*$ means pull-back of a function with respect to the specified morphism, and the subscript $!$ means summations along the fibers.  In other words, for $\E \in \Bun_n X$
\[ (h^{\leftarrow})_{!}(h^{\rightarrow})^{*}(f)(\overline{\E}) = \sum_{\substack{\E' \in \Bun_n X \\  [\E' \xrightarrow[r]{x} \E] 
 \in \mathrm{H}_{x}^{r}(X, \E)}} f(\overline{\E'}).\] 

Finally, the following proposition explains why we call an equivalence class of short exact sequences (with $\E$ fixed)
\[0 \longrightarrow \E' \longrightarrow \E \longrightarrow \cK_{x}^{\oplus r} \longrightarrow 0 \]
a Hecke modification of $\E$.

\begin{prop} The above operator $(h^{\leftarrow})_{!}(h^{\rightarrow})^{*}$ equals $\Phi_{x,r}$.  
\end{prop}

\begin{proof}
    This follows from Weil's Theorem \ref{weil-thm}.
\end{proof}

\begin{cor} As $\E$ runs over $\Bun_n X$, its Hecke modifications  (at $x \in |X|$ with weight $r$) describe the action of $\Phi_{x,r}$ on $ \cA^K.$ \qed
\end{cor}

We end this section by observing that we could consider in the above setting, any torsion sheaf instead of the skyscraper sheaf $\cK_x^{\oplus r}.$  Let $\Coh_{0}X $ be the category of coherent sheaves on $X$ with $0$-dimensional  support. For $\F \in \Coh_{0}X $, 
we define the Hecke operator 
\[ T_{\F}: \cA^K \rightarrow \cA^K\] 
given by 
\[ { \small (T_{\F}f)(\overline{\E}) := \sum_{\substack{\E' \subset \E \\ \E/\E' \cong \F}} f(\overline{\E'}) } \]
where $\E' $ runs over the coherent subsheaves of $\E$ in $\Bun_n X$. Observe that if $\F = \cK_x^{\oplus r}$ above, then $T_{\F}$ equals either $(h^{\leftarrow})_{!}(h^{\rightarrow})^{*}$ or $\Phi_{x,r}$.


\section{Explicit calculation of Hecke modifications} \label{sec-explicitHecke}

This section aims to emphasise the importance of explicit calculations of Hecke modifications. This means the following: given a vector bundle $\E \in \Bun_n X$, a closed point $x \in |X|$ and $r \in \N_{\geq 1}$, find all possible $\E' \in \Bun_n X$ such that $\E' \subset \E$ and $\E/\E' = \cK_{x}^{\oplus r}.$ We note that if $X$ is defined over a finite field, we might also ask about the number of such classes of isomorphisms. We split the situation where $X$ is defined over the complex numbers $\C$ and over a finite field $\Fq$.

\subsection*{Complex geometry} 
Let $X$ be either the projective line or an elliptic curve defined over the complex numbers.
In \cite{boozer-21}, the author describes explicitly all possible Hecke modifications for rank $2$ vector bundles defined over $X$ and then uses this explicit description to give a modular interpretation of the set of all Hecke modifications. We outline this below. 

\begin{rem}
    When the vector bundles under consideration are of rank $2$, the only Hecke modifications which make sense are of weight $1$. If this is the case, we suppress the mention and the notation for the weight.  
\end{rem} 

We first need the concept of (isomorphism classes of) a sequence of Hecke modifications. 

\begin{df}
    Let  $\E \in \Bun_n X$ (holomorphic), a sequence of Hecke modifications  in $\E$ at points $x_1, \ldots, x_n \in |X|$ with weights $r_1, \ldots r_n \in \N_{\geq 1}$, 
    is a collection $\E_i \in \Bun_n X$ and Hecke modifications $[\E_i \xrightarrow[r_i]{x_i} \E_{i-1}]$ for $i = 1, \ldots, n$, where $\E_0:=\E$. We denote such sequence of Hecke modifications by
    \[ [\E_n \xrightarrow[r_n]{x_n} \E_{n-1} \xrightarrow[]{} \cdots \xrightarrow[]{} \E_{1} \xrightarrow[r_1]{x_1} \E].\] 
\end{df}

\begin{df} Let $\E, \E_1, \ldots, \E_n, \E_1', \ldots, \E_n' \in \Bun_n X$. Two sequences of Hecke modifications in $\E$ 
\[ [\E_n \xrightarrow[r_n]{x_n} \E_{n-1} \xrightarrow[]{}  \cdots \xrightarrow[]{} \E_{1} \xrightarrow[r_1]{x_1} \E] \quad  \text{ and }  \quad [\E'_n \xrightarrow[r_n]{x_n} \E_{n-1} \xrightarrow[]{}  \cdots \xrightarrow[]{} \E'_{1} \xrightarrow[r_1]{x_1} \E] \]
 are isomorphic, if there are isomorphisms $\varphi_i: \E_i \to \E'_i$ such that 
\[ \xymatrix@R5pt@C7pt{ \E_n \ar[rr]^{\alpha_n} \ar[dd]^{\cong} && \E_{n-1} \ar[rr]^{\alpha_{n-1}} \ar[dd]^{\cong} && \E_{n-2} \ar[rr]^{\alpha_{n-2}} \ar[dd]^{\cong} && \cdots \ar[rr]^{\alpha_1}  && \E \ar[dd]^{=} \\
&& && && && \\
\E_n' \ar[rr]^{\alpha_n'}  && \E_{n-1}' \ar[rr]^{\alpha_{n-1}'} && \E_{n-2}' \ar[rr]^{\alpha_{n-2}'} && \cdots \ar[rr]^{\alpha_1'} && \E } \]
where $\alpha_i: \E_i \to \E_{i-1}$ and $\alpha_i': \E_i' \to \E_{i-1}'$ stand for the injective morphisms in $[\E_i \xrightarrow[r_i]{x_i} \E_{i-1}]$ and $[\E_i' \xrightarrow[r_i]{x_i} \E_{i-1}'],$ respectively.
\end{df}

 Next we give a correspondence between a sequence of Hecke modification in $\E \in \Bun_2 X$ and a structure of a quasi-parabolic bundle on $\E$. 

Recall that a rank $2$ quasi-parabolic bundle  $(\E, \ell_{x_1}, \ldots , \ell_{x_n})$ over $X$ consists of a vector bundle $\E \in \Bun_2 X$, a choice of distinct points $(x_1, \ldots, x_n) \in X^n$, and for each $i$ a choice of a line $\ell_{x_i}$ in $\E_{x_i} \otimes \C$, for $i=1, \ldots, n$. 
We denote the set of all rank $2$ quasi-parabolic bundles with marked points $(x_1, \ldots, x_n)\in X^n$  by $\mathcal{P}(X,\E,x_1, \ldots, x_n)$, or $\mathcal{P}(X, \E, n)$ when the dependence on $(x_1, \ldots, x_n)$ is not relevant.
 
\begin{rem} Given $\E \in \Bun_2 X$, there is a canonical correspondence between the set of isomorphisms classes of sequences of Hecke modifications in $\E$ at (distinct) points $(x_1, \ldots, x_n) \in X^n$ and $\mathcal{P}(X,\E;x_1, \ldots, x_n)$ given by
\[ [\E_n \xrightarrow[]{x_n} \cdots \xrightarrow[]{x_2} \E_{1} \xrightarrow[]{x_1} \E] \longmapsto (\E, \ell_{x_1}, \ldots, \ell_{x_n}), \]
where $\ell_{x_i}:= \text{im}((\alpha_1 \circ \cdots \circ \alpha_{i})_{x_i}:(\E_i)_{x_i}\otimes \C \to \E_{x_i} \otimes \C)$. 
\end{rem} 

Let $M^s(X,n) $ be the moduli space of stable rank $2$ quasi-parabolic vector bundles over $X$ with trivial determinant and $n$ marked points.
\begin{df}
    Given distinct points $(y_1, \ldots, y_m, x_1, \ldots, x_n) \in X^{m+n}$ we define the space of marked quasi-parabolic bundles $\mathcal{P}(X,m,n)$ to be the set of isomorphism classes of quasi-parabolic bundles  $(\E, \ell_{y_1}, \ldots ,\ell_{y_m}, \ell_{x_1}, \ldots, \ell_{x_n})$ such that $(\E, \ell_{y_1}, \ldots , \ell_{y_m}) \in M^s(X,m)$. For simplicity, we are suppressing the dependence of $\mathcal{P}(X, m, n)$ on $y_1, \ldots, y_m, x_1, \ldots, x_n$ in the notation.
\end{df}

\begin{thm}
    The set $\mathcal{P}(X, m, n)$ naturally has the structure of a complex manifold isomorphic to a $(\P^1)^n$-bundle over $M^s(X, m)$.
\end{thm}

The base manifold $M^s(X, m)$ constitutes of the moduli space over which the isomorphism classes of vector bundles (with marked data) ranges and the $(\P^1)^n$ fibers correspond to a space of Hecke modifications isomorphic to $\P^1$ for each of the points $x_1, \ldots, x_n \in X$.

\begin{df}
    Let $(\E,\ell_{x_1}, \ldots, \ell_{x_n})$ be a quasi-parabolic bundle over $X$. We define the Hecke transform $H(\E, \ell_{x_1}, \ldots,\ell_{x_n})$ of $\E$ to be the sub-bundle $\F \subseteq \E$ whose set of sections over an open set $U \subseteq X$, is constructed as follows
    \[\F(U) := \big\{ s \in \E(U) \;|\; \text{ if } x_i \in U, \text{ then } s(x_i) \in \ell_{x_i} \ \text{ for } \ i=1, \ldots, n\big\}.\]
    In particular, $H(\E, \ell_{x_1}, \ldots,\ell_{x_n})$ is isomorphic to $\E_n$, the first vector bundle in the sequence of Hecke modifications corresponding to $(\E, \ell_{x_1}, \ldots,\ell_{x_n})$.
\end{df}

\begin{df}
    We define $\mathcal{P}_M(X,m,n)$
to be the subset of $\mathcal{P}(X, m, n)$ consisting of points $(E, \ell_{y_1}, \ldots, \ell_{y_m}, \ell_{x_1}, \ldots ,\ell_{x_n})$ such that $H(E,\ell_{x_1}, \ldots ,\ell_{x_n})$ is a semi-stable vector bundle.
\end{df}

\begin{rem}
    As semi-stability is an open condition, the set $\mathcal{P}_M(X,m,n)$
is an open sub-manifold of $\mathcal{P}(X, m, n).$
\end{rem}

In this context, the explicit description of the Hecke modifications plays the following role. In \cite[Sec. 4 and 5]{boozer-21}, the author lists the Hecke modifications of all rank $2$ vector bundles over $X$. The description is applied to study the (in)stability of the Hecke modifications and, among other things,  is applied to show the following lemma.

\begin{lemma}
    Let $(\E, \ell_{x_1}, \ldots, \ell_{x_n}) \in \mathcal{P}(X,\E,x_1, \ldots, x_n)$ be such that $\E$ is semi-stable. If $H(\E, \ell_{x_1}, \ldots, \ell_{x_n})$ is semi-stable, then $(\E, \ell_{x_1}, \ldots, \ell_{x_n})$ is semi-stable.
\end{lemma}

\begin{proof}
    See \cite[Lemmas 4.17 and 5.30]{boozer-21}.
\end{proof}

The lemma yields the following theorem, which is \cite[Thms. 4.19 and 5.31]{boozer-21}. 

\begin{thm} \label{thm-boozer}
    There is a canonical open embedding  \[ \mathcal{P}_M (X,m,n) \hookrightarrow M^s(X,m+n). \]
\end{thm}

With the above discussion, we would like to highlight the importance of describing explicit Hecke modifications in the context of classical/complex algebraic geometry. Note that in \cite{boozer-21}, Boozer is actually interested in using $ \mathcal{P}_M (X,m,n) $ to construct the symplectic Khovanov homology of $n$-stranded links in lens spaces.


\subsection*{Arithmetic geometry} Here we consider again the case where $X$ is defined over a finite field $\Fq.$ We fix $x \in |X|$, $n \in \Z_{> 0}$ and $r \in \Z$ such that $0 \leq r \leq n$. In this context, as we saw before, knowing explicitly the Hecke modifications is equivalent to knowing the action of Hecke operators on the space of unramified automorphic forms. In the following, we report what is known in this case and what kind of results can be derived from Hecke modifications. 

In \cite{zagier-81}, Zagier observes that if the kernel of certain Hecke operators on automorphic forms (defined over $\Q$) turns out to be an unitarizable representation, a formula of Hecke implies the Riemann hypothesis. Zagier calls the elements of this kernel \textit{toroidal automorphic forms} (see  \cite{oliver-toroidal} for a precise definition).  Moreover, Zagier asks what happens if $\Q$ is replaced by a global function field, e.g., $F$ the function field of $X$. Furthermore, Zagier remarks that the space of unramified toroidal automorphic forms can be expected to be finite dimensional. 

Motivated by Zagier's question, Lorscheid introduced in \cite{oliver-graphs} the so-called \textit{Graphs of Hecke operators} over $\PBun_2 X$. These graphs are subsequently considered in \cite{alvarenga19} over $\PBun_n X$, for every $n \geq 1$. When the Hecke operator under considerations is $\Phi_{x,r}$, these are graphs whose set of vertices is $\PBun_n X$ and whose edges are given as follows: for $\overline{\E}, \overline{\E'} \in \PBun_n X$, there is an edge from $\overline{\E}$ to $\overline{\E'},$ if $\E'$ is a Hecke modification of $\E$ at $x$ with weight $r$, i.e.,  if $[\E' \xrightarrow[r]{x} \E] \in \mathrm{H}_{x}^{r}(X, \E)$. 

Moreover, since $X$ is defined over a finite field, there are finitely many extensions of $\cK_{x}^{\oplus r}$ by $\E$. Thus, we might consider the number of classes of isomorphism of such Hecke modification. Namely, the number of $[\E'' \xrightarrow[r]{x} \E] \in \mathrm{H}_{x}^{r}(X, \E)$ such that $\overline{\E''}\cong \overline{\E'}.$ We denote this number by $m_{x,r}(\overline{\E},\overline{\E'}).$ Then Theorem \ref{thm-grass} yields:

\begin{lemma} Let $\E,\E' \in \Bun_n X$, then 
\[ \sum_{\E' \in \Bun_n X} m_{x,r}(\overline{\E},\overline{\E'})= \sum_{\E' \in \Bun_n X} m_{x,r}(\E,\E') = \# \mathrm{Gr}(n-r,n)(\mathbb{F}(x)) \]
and
\[ \sum_{\E \in \Bun_n X} m_{x,r}(\overline{\E},\overline{\E'}) = 
\sum_{\E' \in \Bun_n X} m_{x,r}(\E,\E') = \# \mathrm{Gr}(r,n)(\mathbb{F}(x)), \]
for any $x \in |X|$ and $0 \leq r \leq n.$
\end{lemma}

Graphs of Hecke operators are weighted and oriented graphs whose set of vertices is $\PBun_n X$ and whose edges and weights are given by the Hecke modifications of the vertices. 

\begin{rem}
    We note that one could also define the former graphs considering $\Bun_ X$ as the vertex set. In this case, the graphs obtained from $\PBun_n X$ are projections of the graphs obtained from $\Bun_n X.$  Except for loops, which appear if two adjacent vertices are identified, locally the graphs obtained from $\Bun_ X$ and $\PBun_ X$ are the same. 
\end{rem}

For $n=2$ i.e., over $\Bun_2 X$, Lorscheid describes explicitly the Hecke modifications of every vector bundle when $X$ is the projective line (cf.\ \cite[App. A]{oliver-graphs}) or an elliptic curve (cf.\ \cite[Thm. 3.1]{oliver-elliptic}). These explicit calculations allow him to show, among other things, the following results  which answer some questions posed by Zagier. 

\begin{thm}\cite[Thm. 10.9]{oliver-graphs} The space of unramified toroidal automorphic forms for a global function field is finite dimensional.
\end{thm}

\begin{thm}\cite[Thm. 7.7]{oliver-toroidal} There are no nontrivial unramified toroidal automorphic forms for rational function fields. 
\end{thm}

See \cite[Thm. 5.5]{alvarenga-pereira-23} for the rank $3$ version of the above theorem. 

\begin{thm}\cite[Thm. 7.12]{oliver-elliptic} Let $F$ be the function field of an elliptic curve over a finite field with $q$ elements and class number $h$, let $s+\frac{1}{2}$ be a zero of the zeta function of $F$. If the characteristic is not $2$ or $h\neq q+1$, the space of unramified toroidal automorphic forms is one dimensional and spanned by the Eisenstein series of weight $s$. 
 \end{thm}

Next, we connect the problem of explicit calculations of Hecke modifications with the calculation of certain products in the 
Hall algebra of $X$.

Recall that $X$ is a smooth projective curve of genus $g(X)$ defined over the finite field $\Fq$, and that $\mathrm{Coh}(X)$ denotes the category of coherent sheaves on $X$. 
The category $\mathrm{Coh}(X)$ is never of finite type, i.e., there always are infinitely many indecomposable objects. However, due to the following theorem of Serre, it is a finitary category (i.e., a small abelian category such that, for any two objects $M,N$, $\# \Hom(M,N)< \infty$ and $\# \Ext^1(M,N)< \infty$).

\begin{thm} \label{serrethem} Let $\mathbb{F}$ be a field and $X$ denote a projective variety over $\mathbb{F}$. The category $\mathrm{Coh}(X)$ is a Hom-finite, abelian and noetherian. Moreover, if $X$ is smooth, then 
\[ \dim_{\mathbf{k}}\mathrm{Ext}^{i}(\F,\G) < \infty \text{ for every }i \geq 0 \quad \text{ and } \quad  \mathrm{gldim}\mathrm{Coh}(X) =\dim X.\] 
for any $\F,\G$ objects in $\mathrm{Coh}(X)$.

\begin{proof} See either \cite[Corollary 3.2.3]{grothendieck-61} or the original paper \cite{serre-55}.  
\end{proof}
\end{thm}

The previous theorem allows one to define the Hall algebra of $X$. 


\begin{df} Fix a square root $v$ of $q^{-1}$. Let $\mathbf{H}_X$ be the $\C$-vector space 
\[ \mathbf{H}_X : = \bigoplus_{\F \in \mathrm{Coh}(X)} \C \cdot \F, \]
where we abuse notation and write $\F \in \mathrm{Coh}(X)$ meaning the object in $\mathrm{Coh}(X)$ is given by the isomorphism class of $\F$. 
Given a triple $(\F,\G,\mathcal{H})$ of coherent sheaves on $X$, define 
\[  \text{ $ h_{\F,\G}^{\mathcal{H}} := \frac{\# \big\{ 0 \longrightarrow \G \longrightarrow \mathcal{H} \longrightarrow \F \longrightarrow 0 \big\}}{\#\mathrm{Aut}(\F) \#\mathrm{Aut}(\G)},$}\]
which is finite since $\mathrm{Coh}(X)$ is a finitary category, Theorem \ref{serrethem}. For $\F,\G \in \mathrm{Coh}(X)$, let 
\[ \langle \F,\G \rangle = (1-g)\rk(\F)\rk(\G) + rk(\F) \deg(\G) - \rk(\G)\deg(\F) \]
be the (additive version) Euler form. With the following product, 
\[\F*\G := v^{- \langle \F,\G \rangle} \sum_{\mathcal{H}} h_{\F,\G}^{\mathcal{H}} \mathcal{H},\]
$\mathbf{H}_X $ has the structure of an associative algebra called the \emph{Hall algebra} of $X$ (or of $\mathrm{Coh}(X)$).
\end{df}

The theory of Hall algebras was first discussed in 1901 (in an elementary version) by Ernest Steinitz in  \cite{steinitz-1901}, where he defines what is nowadays called the \textit{Hall polynomials}. These algebras remained forgotten until brought to light by Karsten Johnsen in 1982, in the survey \cite{karsten}. Half a century after Steinitz, Philip Hall ``rediscovered'' the theory of Hall polynomials and Hall algebras in the survey \cite{hall}. However, Hall did not publish anything more than a summary of his theory. The first complete work on Hall algebras, with definitions and proofs, is due to Ian G. Macdonald, in his book \cite{macdonald}. In the early 1990's,  
 the interest in Hall algebras increased after Ringel formalized the notion of a Hall algebra associated to a finitary category and associated it with quantum groups, see  \cite{ringel2}, \cite{ringel1} and \cite{ringel3}. Motived by some questions from  number theory (e.g., Langlands program), Kapranov inverstigated in \cite{kapranov} the Hall algebra of a curve $\mathbf{H}_{X}:= \mathbf{H}_{\mathrm{Coh}(X)}$, as defined above. After Kapranov, many others authors have worked on the theory of Hall algebra of a curve. In the end of 1990s, Baumann and Kassel  published a paper \cite{kassel} concerning the Hall algebra of the projective line. Recently, Schiffmann has published several works concerning the Hall algebras of curves, cf. \cite{olivier-elliptic1} (joint with Burban), \cite{olivier-elliptic2} and \cite{olivier-curve} (joint with Kapranov and Vasserot).

 \begin{rem}
    The Hall algebra might be defined over any finitary category, see \cite{olivier-12}. Moreover, one can define the Hall algebra of a curve defined over an arbitrary field. This is first introduced by Lusztig in \cite{lusztig-91}, where the author replaces $h_{\F,\G}^{\mathcal{H}}$ by the Euler characteristic of the constructible space of all such objects (i.e., the space of all subobjects of $\cH$ of type $\cG$ and cotype $\F$). This variant of the Hall algebra is known as \emph{$\chi$-Hall algebra}.
    
    In \cite{XXZ}, Xiao, Xu and Zhang have given the definition of the $\chi$-Hall algebra of a triangulated category of arbitrary finite global dimension, e.g., for $\mathrm{Coh}(\mathbb{P}^2)$ or coherent sheaves in singular curves. The $\chi$-Hall algebra appears also in \cite{joyce-07} and \cite{brid-12}. 
\end{rem}

Let us return to the problem of describing the Hecke modifications. The link between that and the Hall algebras is given by the following proposition. 

\begin{prop}
    For a fixed $\E \in \Bun_nX$, the Hecke modifications $[\E' \xrightarrow[x]{r} \E]$  are completely determined by 
\[\mathcal{K}_{x}^{\oplus r} *\E' = v^{- \langle \F,\G \rangle} \sum_{\F} h_{\mathcal{K}_{x}^{\oplus r}\E' }^{\F}\; \F\]
where $\E'$ runs through $\mathrm{Bun}_n X.$
Moreover, $h_{\mathcal{K}_{x}^{\oplus r}\E' }^{\E}$ is the number of classes of isomorphism of Hecke modifications $[\E'' \xrightarrow[r]{x} \E]$ with 
$\E'' \cong \E'$, i.e., $h_{\mathcal{K}_{x}^{\oplus r}\E' }^{\E} = m_{x,r}(\E,\E')$. 
\end{prop}

\begin{proof}
    See \cite[Lemmas 2.1 and 2.4]{alvarenga20}. 
\end{proof}

In the following, we consider $X$ to be either the projective line or an elliptic curve. We exhibit two examples of how to describe Hecke modifications using Hall algebras and report what is known.  

\begin{ex}  Let $x \in |\P^1|$ of degree $5$ and $\E := \cO \oplus \cO$. We describe the Hecke modifications of $\E$ at $x$ of weight $1$. 

Since $\deg \cO(-5) < \deg \cO $, we first observe that $\cO(-5) \oplus \cO = \cO(-5) * \cO$, where the product is taken in the Hall algebra $\mathbf{H}_{\P^1}$. Thus
\[
    \cK_x * \cO(-5) \oplus \cO  = 
    (\cK_x * \cO(-5)) * \cO. 
\]
From \cite[Thm. 13]{baumann-kassel-01},  
\[ \cK_x * \cO(-5) = \cO +  q^5 \;
    \cO(-5) \oplus \cK_x. \]
Thus 
\[
    \cK_x * \cO(-5) \oplus \cO  = 
   \cO * \cO +  q^5 
    (\cO(-5) \oplus \cK_x) * \cO.
\]
Again from \cite[Thm. 13]{baumann-kassel-01},
\[ 
 \cO * \cO = (q+1)  \cO \oplus \cO, \quad
 \cO(-5) \oplus \cK_x = \cO(-5) * \cK_x, \quad
 \text{ and } \quad 
 \cK_x * \cO = \cO(5) +  q^5 \;
    \cO \oplus \cK_x. 
\]
Therefore,
\[ \cK_x * \cO(-5) \oplus \cO = 
(q+1)  \cO \oplus \cO + 
\cO(-5) \oplus \cO(5) +  q^5 \;
    \cO(-5) \oplus \cO \oplus \cK_x.
\]
Hence we have the following Hecke modifications: $[\cO(-5) \oplus \cO \xrightarrow[]{x} \cO \oplus \cO]$ with multiplicity $q+1$ and  $[\cO(-5) \oplus \cO \xrightarrow[]{x} \cO(-5) \oplus \cO(5)]$ with multiplicity $q^5$. 

As one can observe from Theorem \ref{thm-grass}, the sum of the multiplicities of the Hecke modification of $\cO \oplus \cO$ at $x$ with weight $1$ equals the number of rational points at $\P^{1}(\mathbb{F}_{q^5})$ i.e., $q^5-1$. Thus, $[\cO(-5) \oplus \cO \xrightarrow[]{x} \cO \oplus \cO]$ is not the unique Hecke modification of $\cO \oplus \cO$ at $x$ with weight $1$. 

We use again \cite[Thm. 13]{baumann-kassel-01} to conclude that 
\begin{align*}
\cK_x * \cO(-3) \oplus \cO(-2) & = 
q^5 \;  \cO(2) \oplus \cO(-2) + 
q^3(q^2 -1)\;\cO(-1) \oplus \cO(1)  \\
& + q^3(q^2-1) \; \cO \oplus \cO  + 
q^5 \; \cO(-3) \oplus \cO(3) + 
q^{10} \; \cO(-3) \oplus \cO(-2) \oplus \cK_x.
\end{align*}
This gives one more Hecke modification of 
$\cO \oplus \cO$, namely
\[ [\cO(-3) \oplus \cO(-2) \xrightarrow[]{x} \cO \oplus \cO] \]
with multiplicity $q^5 - q^3$. 
The last Hecke modification of $\cO \oplus \cO$ is given by the product
\begin{align*}
\cK_x * \cO(-4) \oplus \cO(-1) & =  
\cK_x * (\cO(-4) * \cO(-1)) \\
& = ( \cO(1) + q^5 \;  \cO(-4) \oplus \cK_x)*\cO(-1) \\ 
& = \cO(1) * \cO(-1) + q^5 (\cO(-4) *(\cO(4)+q^5 \; \cO(-1)\oplus K_y)) \\
& = q^3 \;  \cO(-1) \oplus \cO(1) + 
(q^3-q)\;  \cO \oplus \cO + q^5 \;  \cO(-4) \oplus \cO(4) \\
& \quad + q^{10} \;  \cO(-4) \oplus \cO(-1) \oplus \cK_x.
\end{align*}
Hence, the last Hecke modification of $\cO \oplus \cO$ at $x$ with weight $1$ is 
\[ [\cO(-4) \oplus \cO(-1) \xrightarrow[]{x} \cO \oplus \cO] \]
with multiplicity $q^3 - q$. 
\end{ex} 

In a forthcoming work \cite{alvarenga-moco-25}, the first and third authors descibe, among other things, explicitly the Hecke modifications for $\E \in \Bun_n \P^1$, for every $n \in \N$. Moreover, the case where $\P^1$ is defined over the complex field is also considered. 

In \cite{alvarenga20} the first author treats the case where $X$ is an elliptic curve. Using the structural results in \cite{olivier-elliptic1} and \cite{dragos-13}, he produces an algorithm to calculate the Hecke modifications of any vector bundles defined over $X$. In the following, we give an example of how the algorithm runs. 

\begin{ex}
    Let $X$ be the elliptic curve defined over $\mathbb{F}_2$, given by $y^2 + y = x^3 + x+1.$  We observe that $X$ has only one rational point, say $x_0 \in X(\mathbb{F}_2)$. 

    We first observe that by Atyiah's classification of vector bundles on elliptic curves (see \cite{atiyah-57}), any indecomposable vector bundle $\E$ on $X$ is defined by its rank $n$, degree $d$, a closed point $x \in |X|$ and an integer $\ell \in \Z$ such that $\ell |x| = \gcd(n,d)$. Therefore, we denote $\E$ by $\E_{(x,\ell)}^{(n,d)}$. We also refer to \cite{burban-07} and \cite{olivier-elliptic1} for Atiyah's classification over any base field. 

    Let $x \in |X|$ of degree $2$ and $\E' = \E_{(x,1)}^{(2,0)}$. Next, we calculate all the vector bundles $\E \in \Bun_2 X$ that are Hecke modified in $\E'$ at $x$ with weight $1$, i.e., the Hecke modifications $[\E' \xrightarrow[]{x} \E]$ for $\E'$ as above. In terms of the Hall algebra of $X$, this means to calculate the vector bundles which appears in the product $\cK_x * \E'$. The strategy is to write both  $\E'$ and $\cK_x $ in terms of some generators of $\mathbf{H}_X$ that are well understood and allow us to do explicit multiplication. 

    For $m \in \N$ and $z \in |X|$ such that $|z|\; | \; m$, we define the elements $T_{(0,m),z} \in \mathbf{H}_X$ by
\[T_{(0,m),z} := \frac{[m]|z|}{m} \sum_{|\lambda|=m/|z|} 
\prod_{i=1}^{l(\lambda)-1}(1-q_{z}^{i/2}) \mathcal{K}_{z}^{(\lambda)},\]
where  $[m] := \tfrac{q^{-s/2} - q^{-s/2}}{q^{-1/2} - q^{1/2}}$; $\lambda$ is a partition, $|\lambda|$ its weight, $l(\lambda)$ its length; and $\mathcal{K}_{z}^{(\lambda)} \in \Coh(X)$ is the unique torsion sheaf supported at $z$ and given by $\lambda$. For example, if $\lambda = 1$, $\mathcal{K}_{z}^{(\lambda)} = \cK_z$ is the skyscraper sheaf at $z$ and if $\lambda = (1, \ldots, 1)$, then $\mathcal{K}_{z}^{(\lambda)} = \cK_{z}^{\oplus r}$, where $r = l(\lambda)$. 

The advantage to working with $T_{(0,m),z}$ is that they can be seen as elements in the Macdonald's ring of symmetric functions, where the computation can be done explicitly, see \cite[subsection 4.2]{olivier-elliptic1}. Via Atiyah's classification, we can extend the definition from $T_{(0,m),z}$ to $T_{(n,d),z}$, for every $n \in \N$ and $d \in \Z$, parametrizing the rank and degree of any coherent sheave on $X$. Thus, 
\begin{equation} \label{eq-example}
   \E' = \tfrac{1}{[2]} T_{(2,0),x} \quad \text{ and } \quad  \cK_x = \tfrac{1}{[2]} T_{(0,2),x}.
\end{equation}
In order to write both $\E'$ and $\cK_x$ in terms of the above cited basis of $\mathbf{H}_X$, we need to consider: 
\begin{itemize}
    \item $X_n := X \otimes_{\mathbb{F}_q} \mathbb{F}_{q^n}$ the base field extension of degree $n$ over $X$; 

    \item $\mathrm{Pic}^{0}(X)$ (resp. $\mathrm{Pic}^{0}(X_n)$) the group of degree zero divisors on $X$ (resp. $X_n$); 

    \item $\widehat{\mathrm{Pic}^{0}(X)}$ (resp. $\widehat{\mathrm{Pic}^{0}(X_n)}$) the character group of $\mathrm{Pic}^{0}(X)$ (resp. $\mathrm{Pic}^{0}(X_n)$); and

    \item $\mathrm{Fr}_{X,n} : X_n \rightarrow X_n$ the Frobenius on $X_n$ relative to $X$ (i.e., induced by the $q$-power on $\cO_X$). 
\end{itemize}
The choice of $x_0 \in X(\Fq)$, allows identifying $\mathrm{Pic}^d X_n$, the divisors of degree $d$ on $X_n$, with $\mathrm{Pic}^0 X_n$. Thus, we can extend a character $\rho$ from $\mathrm{Pic}^0 X_n$ to $\mathrm{Pic}^d X_n$ placing $\rho(x_0)=1.$  
Let $\mathrm{Fr}_{X,n}^{*}$ be the map induced by $\mathrm{Fr}_{X,n}$ on $\widehat{\mathrm{Pic}^{0}(X_n)}$. For  $\widetilde{\rho} \in \widehat{\mathrm{Pic}^{0}(X_n)}/\mathrm{Fr}_{X,n}^{*}$ and a closed point $z \in X$, we define 
\[ \widetilde{\rho}(z) := \frac{1}{n} \sum_{i=0}^{n-1}\rho((\mathrm{Fr}_{X,n}^{*})^{i}(\mathcal{O}_{X_n}(z'))),\]
where $z' \in X_n$ is a closed point that sits above $z$.
For  $\widetilde{\rho} \in \widehat{\mathrm{Pic}^{0}(X_n)}/\mathrm{Fr}_{X,n}^{*}$ and $\mathbf{v} :=(n,d) \in \Z^2$ 
parametrizing the rank and degree of a coherent sheaf on $X$, we define
\[ T_{\mathbf{v}}^{\widetilde{\rho}} := \sum_{x \in |X|} \widetilde{\rho}(x) T_{\mathbf{v},x}.\] 
By definition, $T_{\mathbf{v},x} = 0$, unless $|x|$ divides $\gcd(n,d).$ These $T_{\mathbf{v}}^{\widetilde{\rho}}$ are generators for the  Hall algebra $\mathbf{H}_{X}$, see \cite[Prop. 3.4]{dragos-13}. The following facts highlight the significance of working with these generators : (i) if 
$\widetilde{\rho} \in \widehat{\mathrm{Pic}^{0}(X_n)}/\mathrm{Fr}_{X,n}^{*}$ and $\widetilde{\sigma} \in \widehat{\mathrm{Pic}^{0}(X_m)}/\mathrm{Fr}_{X,m}^{*}$ are different and not given by the norm of the same character, then
\begin{equation} \tag{Step 1} \label{step1}
    T_{\mathbf{v}}^{\widetilde{\rho}} * T_{\mathbf{v'}}^{\widetilde{\sigma}} = T_{\mathbf{v'}}^{\widetilde{\sigma}} * T_{\mathbf{v}}^{\widetilde{\rho}}, 
\end{equation}
where $v' \in \Z^2$ parametrizing the rank and degree of a coherent sheaf; (2) it is always possible to write the product  $T_{\mathbf{v}}^{\widetilde{\rho}} * T_{\mathbf{v'}}^{\widetilde{\rho}}$ as 
\begin{equation} \tag{Step 2} \label{step2}
    T_{\mathbf{v_1}}^{\widetilde{\rho}} \cdots T_{\mathbf{v_k}}^{\widetilde{\rho}} 
\end{equation}
where the slopes of $v_1, \ldots, v_k$ are not decreasing, see \cite[Def. 3.11 and Thm. 5.2]{dragos-13}.  This is fundamental since over elliptic curves, if $\F,\G \in \Coh X$ of slopes $\mu$ and $\nu$ (resp.) such that $\mu < \nu$, then $\Ext^1(\F,G)=0.$ 

We return to our situation, where $n=2$ above. Let $X(\mathbb{F}_2) = \{x_0\}$ and $X(\mathbb{F}_4) = \{x_0, x_1, x_2, y_1, y_2 \}$, where 
$x_1 := (0: \alpha : 1),\;  x_2 := (0: \alpha^2 : 1),\; y_1 := (1: \alpha : 1)$, $y_2 : = (1: \alpha^2 : 1)$ and $\alpha \in \mathbb{F}_4$ is such that $\mathbb{F}_4 = \mathbb{F}_2(\alpha)$. Hence 
\[  \widehat{\mathrm{Pic}^{0}(X)} := \{\rho_0\} \quad \text{ and } \quad 
 \widehat{\mathrm{Pic}^{0}(X)} := \{\rho_0, \rho_1, \rho_2, \rho_3, \rho_4 \}  \]
where $\rho_0$ is the trivial character, $\rho_i(x_1) = \zeta^i$ for $i=0, \ldots, 4$ and $\zeta \in \C$ is a  primitive quintic root of unity. 

We let $x_1, x_2 \in |X_2|$ be the two degree-$1$ points of $X_2$ sitting above $x \in |X|$. Observe that $X$ has just one more closed point $y \in |X|$ of degree $2$, and thus $y_1,y_2 \in |X_2|$ sit above it. 

As previously mentioned, the strategy is to write $\E'$ and $\cK_x$ in terms of the generators $T_{\mathbf{v}}^{\widetilde{\rho}}.$ Observe that 
\[ \sum_{\rho \in \widehat{\mathrm{Pic}^{0}(X_2)}} \widetilde{\rho}(y) T_{(0,2)}^{\widetilde{\rho}} = 
T_{(0,2)}^{\widetilde{\rho_0}} + \tfrac{(\zeta + \zeta^4)}{2} T_{(0,2)}^{\widetilde{\rho_1}}
+ \tfrac{(\zeta + \zeta^4)}{2} T_{(0,2)}^{\widetilde{\rho_4}} 
+ \tfrac{(\zeta^2 + \zeta^3)}{2} T_{(0,2)}^{\widetilde{\rho_2}}
+ \tfrac{(\zeta^2 + \zeta^3)}{2} T_{(0,2)}^{\widetilde{\rho_3}},
\]
while 
\begin{align*}
    T_{(0,2)}^{\widetilde{\rho_0}} & = T_{(0,2),x_0} +  T_{(0,2),x} +  T_{(0,2),y} \\
    T_{(0,2)}^{\widetilde{\rho_1}} & = T_{(0,2),x_0} +  \tfrac{(\zeta + \zeta^4)}{2} T_{(0,2),x} + \tfrac{(\zeta^2 + \zeta^3)}{2} T_{(0,2),y} =  T_{(0,2)}^{\widetilde{\rho_4}} \\
     T_{(0,2)}^{\widetilde{\rho_2}} & = T_{(0,2),x_0} +  \tfrac{(\zeta^2 + \zeta^3)}{2} T_{(0,2),x} + \tfrac{(\zeta + \zeta^4)}{2} T_{(0,2),y} =  T_{(0,2)}^{\widetilde{\rho_3}}. \\
\end{align*}
Thus
\[  \sum_{\rho \in \widehat{\mathrm{Pic}^{0}(X_2)}} \widetilde{\rho}(y) T_{(0,2)}^{\widetilde{\rho}} =  \tfrac{5}{2} T_{(0,2),x} = \tfrac{\#X(\mathbb{F}_4)}{2} T_{(0,2),x}. \]
Identity \eqref{eq-example} yields 
\[  \E' = \tfrac{2}{5 [2]} \sum_{\rho \in \widehat{\mathrm{Pic}^{0}(X_2)}} \widetilde{\rho}(y) T_{(2,0)}^{\widetilde{\rho}} \quad \text{ and  } \quad 
\cK_x = \tfrac{2}{5 [2]} \sum_{\rho \in \widehat{\mathrm{Pic}^{0}(X_2)}} \widetilde{\rho}(y) T_{(0,2)}^{\widetilde{\rho}}. \] 
From \cite[Lemma 2.10]{dragos-13}, the Lie bracket $\big[\cK_x, \E'\big]$ gives us all the vector bundles which appear in the product $\cK_x * \E'$. By \eqref{step1}
\[ \big[\cK_x, \E'\big] = \tfrac{4}{25 [2]^2} 
\left( \rho_0(y)^2 \big[T_{(0,2)}^{\widetilde{\rho_0}}, T_{(2,0)}^{\widetilde{\rho_0}} \big]   + 
4 \rho_1(y)^2 \big[T_{(0,2)}^{\widetilde{\rho_1}}, T_{(2,0)}^{\widetilde{\rho_1}} \big]   +  
4 \rho_2(y)^2 \big[T_{(0,2)}^{\widetilde{\rho_2}}, T_{(2,0)}^{\widetilde{\rho_2}} \big]   \right). \]
From, \cite[Thm. 5.2 ]{dragos-13} i.e., \eqref{step2},
\[ \big[T_{(0,2)}^{\widetilde{\rho_0}}, T_{(2,0)}^{\widetilde{\rho_0}} \big]  = 
\tfrac{c_2^2 (q^{1/2}-q^{-1/2})}{2 c_1} T_{(1,1)}^{\widetilde{\rho_0}} T_{(1,1)}^{\widetilde{\rho_0}} + c_2 (\tfrac{c_2}{c_1}-2) T_{(2,2)}^{\widetilde{\rho_0}},\]
and 
\[ \big[T_{(0,2)}^{\widetilde{\rho_i}}, T_{(2,0)}^{\widetilde{\rho_i}} \big] = 
\tfrac{5[2]}{2q} T_{(2,2)}^{\widetilde{\rho_i}},\]
where $c_i = q^{-1/2} [i] \#X(\mathbb{F}_{q^i})/i$ and $i=1,2.$

Since we have written $\big[\cK_x, \E'\big]$ as sum of products of elements in $\mathbf{H}_X$ ordered by slopes in a non-decreasing order, we might proceed with base change. We observe that
\[ T_{(1,1)}^{\widetilde{\rho_0}} T_{(1,1)}^{\widetilde{\rho_0}}  = T_{(1,1),x_0} T_{(1,1),x_0}, \quad \; T_{(2,2)}^{\widetilde{\rho_0}} =  T_{(2,2),x_0} + T_{(2,2),x} + T_{(2,2),y},\]
and 
\[T_{(2,2)}^{\widetilde{\rho_i}} =  T_{(2,2),x_0} + \rho_i(x) T_{(2,2),x} + \rho_i(y) T_{(2,2),y}\]
for $i=1,2$. Thus, 
\begin{align*}
     \big[\cK_x, \E'\big] & = \tfrac{(q^{-1}- q^{-2})}{2}   T_{(1,1),x_0} T_{(1,1),x_0} + 
     \tfrac{(q^{-2}-q^{-1})}{[2]}  T_{(2,2),x_0} + \tfrac{q^{-3/2}}{[2]} T_{(2,2),x} + 
     \tfrac{q^{-2}}{[2]} T_{(2,2),y}.
\end{align*}
By definition 
\[ T_{(2,2),x_0} = \tfrac{[2]}{2} \left(  \E_{(x_0,2)}^{(2,2)} + (q+1) \E_{(x_0,1)}^{(1,1)} \oplus \E_{(x_0,1)}^{(1,1)} \right)\]
while
\[   T_{(2,2),x}  = [2]  \E_{(x,1)}^{(2,2)}, \quad  \text{ and } \quad 
T_{(2,2),y}  = [2]  \E_{(y,1)}^{(2,2)}.\]
The missing product $T_{(1,1),x_0} T_{(1,1),x_0}$ can be realized in the Macdonald ring of symmetric functions, as notice in \cite[subsection 4.2]{olivier-elliptic1}. This computation can be done using \cite{sagemath}, where we obtain 
\[ T_{(1,1),x_0} T_{(1,1),x_0} = \E_{(x_0,2)}^{(2,2)} + (q+1) \E_{(x_0,1)}^{(1,1)} \oplus \E_{(x_0,1)}^{(1,1)}.  \]
Hence, 
\[
     \big[\cK_x, \E'\big]  = (1 - q^{-1}) \E_{(x_0,1)}^{(1,1)} \oplus \E_{(x_0,1)}^{(1,1)} + (q^{-2}+ q^{-1}) \E_{(x,1)}^{(2,2)} + q^{-2} \;\E_{(y,1)}^{(2,2)}.
\]
Therefore, 
\[ [\E' \xrightarrow[]{x} \E_{(x_0,1)}^{(1,1)} \oplus \E_{(x_0,1)}^{(1,1)}], \quad 
 [\E' \xrightarrow[]{x} \E_{(x,1)}^{(2,2)}], \quad \text{and} \quad 
 [\E' \xrightarrow[]{x} \E_{(y,1)}^{(2,2)}] \]
 are all Hecke modifications $[\E'' \xrightarrow[]{x} \E]$ with $\E'' \cong \E'.$ 
The multiplicities are obtained by multiplying above identity by the Euler form of $\cK_x$ and $\E'$, i.e., multiplying by $q^{2}$. 
\end{ex}

\section{Hecke modifications for rank $2$ vector bundles: elementary transformations} \label{sec-elementarytransf}

As seen in the last section, there is a correspondence between Hecke modifications and quasi-parabolic bundles. In this section, we descibe how for rank $2$ vector bundles, this can be viewed as birational morphisms of ruled surfaces.

 Let $(\E, \ell_{p_1}, \cdots , \ell_{p_n})$ be a
 rank $2$ quasi-parabolic bundle on a curve $X$.  Let $T$ be a subset of  $\{1, \dots , n\}$  of cardinality $t \geq 0$. We can define a Hecke modification $\E'$ of the vector bundle $\E$ via the following sequence: 
\[ 0 \to \E' \xrightarrow{\alpha} \E \to \bigoplus_{i \in T} (\E_{x_i}/\ell_{x_i})\otimes\mathcal{O}_{x_i} \to 0. \]

 If $i \notin  T$, then $\alpha_{x_i}: \E_{x_i} \to \E^{\prime}_{x_i}$ is an isomorphism and $(\ell^{\prime}_{x_i})^{-1} = (\alpha_{x_i})^{-1}\ell_{x_i}$ is the linear subspace at the point $x_{i}$. On the other hand if $i \in T$, then $\ell_{p_i}^{\prime} = \mathrm{ker}(\alpha_{i})$ gives the quasi-parabolic structure at the point $x_i$. 
 
 We now see how obtaining a new quasi-parabolic vector bundle $(\E^{\prime}, \ell^{\prime}_{x_1}, \cdots , \ell^{\prime}_{x_n})$ from the quasi-parabolic bundle $(\E, \ell_{x_1}, \cdots , \ell_{x_n})$ can be seen as a birational morphism of ruled surfaces. Let us recall the following definition.

  \begin{df}
      A \emph{ruled surface} is a surface $S$ together with a surjective morphism $\pi: S\to X$ to a non-singular curve $X$ such that the fiber $S_{x}$ is isomorphic to $\mathbb{P}^1$ for every point $x \in X$ and such that $\pi$ admits a section (i.e., a morphism $\sigma: X \to S$ such that $\pi \bullet \sigma = \mathrm{id}_{X}$).
  \end{df} 

  We recall the following key result. We refer the reader to \cite{hartshorne} and \cite{friedman2012algebraic} for proofs and related results.
  
 \begin{prop}\label{projsurface}
     If $\pi: S \to X $ is a ruled surface, then there exists a locally free sheaf $\E$ of rank $2$ on $X$ such that $S\simeq \mathbb{P}(\E)$ over $X$, where $\mathbb{P}(\E)$ is the projectivization of $\E$ (see \cite[Chapter $II.7$]{hartshorne} for definition of $\mathbb{P}(\E)$). Conversely, every such $\mathbb{P}(\E)$ is a ruled surface over $X$. If $\E$ and $\E^{\prime}$ are two locally free sheaves of rank $2$ on $X$, then $\mathbb{P}(\E)$ and $\mathbb{P}(\E^{\prime})$ are isomorphic as ruled surfaces over $X$ if and only if there is an invertible sheaf $\mathcal{L}$ on $X$ such that $\E^{\prime} \simeq \E \otimes \mathcal{L}$.
  \end{prop}

 For simplicity of notation, let us assume that $(\E, \ell_{x})$ is a rank $2$ quasi-parabolic vector bundle with quasi-parabolic structure at only one point $x$. 
 By Proposition \ref{projsurface} considering the vector bundle $\E$ as a locally free sheaf, we can view its projectivization $\mathbb{P}(\E)$  as a ruled surface over the curve $X$, say $\pi: \mathbb{P}(\E) \to X$. We denote by $L$ the projectivization of the $2$-dimensional fiber of $\E$ and by $P$ the projectivization of the quasi-parabolic structure $\ell_{x}$. Observe that $L$ is a line  and contains $P$ as a point.

 Let 
 \[f: \widetilde{\mathbb{P}(\E)} \to \mathbb{P}(\E)\] 
 be the blow-up of $\mathbb{P}(\E)$ at the point $P$. 
 We observe that the strict transform $\tilde{L}$ of $L$ on $\widetilde{\mathbb{P}(\E)}$ is isomorphic to $\mathbb{P}^1$. Since $P$ is a nonsingular point of $L$, it has multiplicity $1$. By \cite[Chapter V, Proposition $3.6$]{hartshorne}, $\tilde{L} \sim f^{*}L - E$, where $E$ is the exceptional divisor. Furthermore by \cite[Chapter V, Proposition $2.3$]{hartshorne}, $L^2 = 0$. Therefore $\tilde{L}^2 = -1$.
 Then by Castelnuovo's contractability theorem, there is a morphism, say $g$ from $\widetilde{\mathbb{P}(\E)}$ to a surface $S$ contracting the $-1$ curve $\tilde{L}$ to a point $Q$. In other words, $g:\widetilde{\mathbb{P}(\E)} \to S$ is the blow-up of $S$ with center $Q$. It is easy to check that $g(E) \simeq \mathbb{P}^1$ and $g(E)^2 = 0$. Note that outside of $L$ and $M$ (respectively), $\mathbb{P}(\E) \simeq S$. Therefore  $\pi^{\prime}: S \to X$ is in fact a morphism. Moreover, the fibers of $\pi^{\prime}$ are all isomorphic to $\mathbb{P}^1$ and the strict transform of the section $D$ is a section $D^{\prime}$ on $S$. Therefore, $\pi^{\prime}:S \to X$ is another ruled surface. 
 
 By the converse of Proposition \ref{projsurface}, we know that any ruled surface is the projectivization of a locally free sheaf. Therefore, "deprojectivizing" $S$ and $Q$, we get  $\E^{\prime}$ to be the locally free sheaf such that $S = \mathbb{P}(\E^{\prime})$ and  $\ell^{\prime}_x$ to be the quasi-parabolic structure such that $Q = \mathbb{P}(\ell^{\prime}_{x})$. 

\begin{df}
 We say the ruled surface $\pi^{\prime}: \mathbb{P}(\E^{\prime}) \to X$ is obtained via an \emph{elementary transformation of $\mathbb{P}(\E)$ with center $P$}. As we can see from the above discussion, this corresponds to the Hecke modification associated to the quasi-parabolic vector bundle $(\E, \ell_x)$ at the parabolic point $x \in X$.
\end{df}

\begin{rem} We make the following remarks:
\begin{enumerate}
     \item The above construction can be done for any number of points and the genus of the curve $X$ does not play a role, only that the curve is nonsingular. 
   \item When $X = \mathbb{P}^1$ is a projective line, using an elementary transformation one can inductively construct rational ruled surfaces $\mathbb{S}_n$, for $n\geq 0$ with $\mathbb{S}_0$ being the smooth quadric $\mathbb{P}^1 \times \mathbb{P}^1$ in $\mathbb{P}^3$. These surfaces are also called \emph{Hirzebruch surfaces}. See Chapter $5$ of \cite{friedman2012algebraic} for their construction.
   \item It is easy to see from the discussion above that elementary transformations are involutions. Suppose $\mathrm{el}_T$ (respectively $\mathrm{el}_R$) denotes the elementary transformation centered at the parabolic points $\{x_i\}_{i \in T}$ (respectively at the parabolic points $\{x_i\}_{i \in R}$) then $\mathrm{el}_{T} \bullet \mathrm{el}_R = \mathrm{el}_{T \cup R \backslash T\cap R}$. The set of elementary transformations form a group.
   \item Elementary transformations have several applications. Some recent examples include \cite{loray-heu-19} where the authors used elementary transformations to study moduli spaces of rank $2$ (quasi)-parabolic vector bundles with logarithmic connections, and \cite{Araujo2021automorphisms} where the authors show that any automorphism of the moduli space of rank $2$, degree $0$ parabolic bundles on $\mathbb{P}^1$ with at least $5$ parabolic points and weights $(1/2)$ at all points is an elementary transformation. 
     \item Lastly, if the rank of the vector bundle is greater than $2$, the projectivization of the vector bundle is no longer a surface so we cannot simply see the process of obtaining a Hecke modification as a birational transformation of the surface using Castelnuovo's contraction theorem. In this case, the geometric interpretation is much more complicated.
\end{enumerate}
    
\end{rem}


\section{Directory for future works} 

We end this article proposing some directions for future works. 

\subsection*{$G$-bundles} 
The full article might be considered replacing $\GL_n$ by other split connected reductive group $G$ i.e., by considering $G$-bundles instead $\GL_n$-bundles. Thus, one might consider investigating the theory developed in this article in this new setting. In particular, considering $G$-bundles for every split connected reductive group is of special interest in the geometric Langlands program.

\subsection*{Level structure} Since $\GL_n$ can be replaced by a split connected reductive group $G$, one might also explore $G$-bundles with level structures. For definition and first properties of vector bundle with level structure, we refer to \cite{seshadri-82}.  In this new scenario, Hecke modifications parametrize the action of ramified Hecke operators on the space of ramified automorphic forms, where the ramification is induced by the level structure. In \cite{alvarenga-bonnel-24}, the authors investigate the scenario where the ramification is given by a closed point of degree one and describe the action of ramified Hecke operators using number-theoretic tools. A version of it using algebraic geometry i.e., Hecke modifications of vector bundles with level structure, would be of interest.  The connection between ramified automorphic forms and $G$-bundles with level structure is explained in \cite[p. 154]{lafforgue-18}. In \cite{nguy-11}, the author establishes Weil's theorem (Theorem \ref{weil-thm}) in this context. 

\subsection*{The Hecke stack} 
In order to keep our exposition elementary, 
we have vaguely mentioned the Hecke stack in Remark \ref{rem-heckestack}. However, a detailed investigation of its algebraic and geometric properties would be an interesting problem.

\subsection*{Quasi-Parabolic structures}
In Section \ref{sec-explicitHecke}, we highlighted the connection between Hecke modifications and quasi-parabolic structures. That correspondence strongly uses the fact that we are working with rank $2$ vector bundles. In upcoming work \cite{alvarenga-kaur-moco-25}, we generalize this correspondence to higher-rank vector bundles. 


\subsection*{Moduli embeddings}  With the above correspondence established -- namely, the connection between Hecke modifications and quasi-parabolic structures for higher rank vector bundles -- one might ask for a generalization of Theorem \ref{thm-boozer}. 

\subsection*{Higher genus curves I} As noted in Section \ref{sec-explicitHecke}, explicit descriptions of Hecke modifications for genus $0$ and $1$ curves are provided in \cite{alvarenga-moco-25} and \cite{alvarenga20}. However, even for rank $2$, the explicit description of Hecke modifications for higher genus curves remains completely open. Based on the low genus cases, it seems important to first develop a suitable classification of vector bundles before performing such modifications. For genus $2$ curves, Newstead provides a classification of stable vector bundles in \cite{newstead-68}. In \cite{desale-ramanan-76}, the authors generalize Newstead's work to hyperelliptic curves. However, the case of classification of stable vector bundles over an arbitrary curve of high genus is wide open.   

\subsection*{Higher genus curves II}
The success of the above item might open an opportunity to investigate the space of automorphic forms for higher genus curves. This means generalizing \cite{oliver-toroidal} and \cite{alvarenga-pereira-23} for higher genus curves. Since the genus $0$ case has trivial Hecke eigenspaces and the genus $1$ is the unique \textit{nontrivial} example known so far,  developing the theory for higher genus curves would be particularly valuable, allowing comparisons with the genus $1$ case. 

\subsection*{Hall algebras} As we explain in Section \ref{sec-explicitHecke}, over a finite field, the Hecke modifications yield some Hall numbers in the Hall algebra of the associated curve.  It would be interesting to investigate the implications of these Hall numbers for the Hall algebra structure, in the spirit of \cite[Sec. 3]{baumann-kassel-01}.  

\subsection*{Higher rank elementary transformations} The geometric interpretation of the Hecke modifications presented in Section \ref{sec-elementarytransf} works for rank $2$ vector bundles. Extending this interpretation to higher rank vector bundles would be another worthwhile effort.  

\section*{Acknowledgments} This work was financed, in part, by the São Paulo Research Foundation (FAPESP), Brasil. Process Number 2022/09476-7. The second author is grateful for a CAPES visiting fellowship that funded a stay at the ICMC-USP in July 2023, when discussions related to this work began.


\end{document}